\RequirePackage[l2tabu,orthodox]{nag} 
\documentclass[a4paper,10pt]{amsart}
\usepackage{etex}
\usepackage[T1]{fontenc}
\usepackage[utf8]{inputenx}
\usepackage{lmodern}
\usepackage[kerning=true,tracking=true]{microtype}
\usepackage{fixltx2e}
\usepackage{lmodern}
\usepackage{graphics}
\usepackage[pdftex,pagebackref]{hyperref}
\hypersetup{
  colorlinks   = true,  
  urlcolor     = black, 
  linkcolor    = black, 
  citecolor    = black  
}

\title[Geometries on tori]%
{Symmetries of holomorphic geometric structures on tori}
\date{April 13, 2010}

\author[S. Dumitrescu]{Sorin Dumitrescu}
\address{Universit\'e C\^ote d'Azur, Universit\'e  Nice Sophia Antipolis, UMR CNRS 7351, Laboratoire J.-A. Dieudonn\'e,   Parc Valrose, 06108 NICE Cedex 02, France}
\email{dumitres@unice.fr}
 
\author[B. McKay]{Benjamin McKay}
\address{University College Cork, Cork, Ireland}
\email{b.mckay@ucc.ie}

\date{\today}

\usepackage{microtype}
\usepackage{xspace}
\usepackage{amsfonts}
\usepackage{amssymb}
\usepackage{mathtools}
\usepackage{braket}
\usepackage{varioref}
\usepackage{xstring}
\usepackage{array}
\usepackage{ragged2e}
\usepackage{longtable}
\usepackage{multirow}
\usepackage{booktabs}
\usepackage[mathscr]{euscript}
\usepackage{mathrsfs}
\usepackage{tikz-cd}
\newtheorem{theorem}{Theorem}
\newtheorem{corollary}{Corollary}
\newtheorem{lemma}{Lemma}
\newtheorem{proposition}{Proposition}
\theoremstyle{remark}

\newtheorem{example}{Example}

\newcommand*{\pr}[1]{\ensuremath{\left(#1\right)}}
\newcommand*{\curly}[1]{\ensuremath{\left\{#1\right\}}}
\newcommand*{\of}[1]{\ensuremath{\!\pr{#1}}}
\newcommand*{\C}[1]{\ensuremath{\mathbb{C}^{#1}}}

\newcommand*{\Proj}[1]{\ensuremath{\mathbb{P}^{#1}}}

\newcommand*{\GL}[1]{\ensuremath{\operatorname{GL}\of{#1}}}

\newcommand*{\PSL}[1]{\ensuremath{\mathbb{P}\!\operatorname{SL}\of{#1}}}

\newcommand*{\Lm}[2]{\ensuremath{\Lambda^{#1}\of{#2}}}

\DeclareMathOperator{\Ad}{Ad}

\newcommand*{\map}[3][:]%
{\ensuremath{\IfStrEq{#1}{:}{}{{#1} \colon}{#2} \to {#3}}}
\newcommand*{\mapto}[3][:]%
{\ensuremath{\IfStrEq{#1}{:}{}{{#1} \colon}{#2} \mapsto {#3}}}
\newcommand*{\Hom}[2]{\ensuremath{\operatorname{Hom}\of{{#1},{#2}}}}
\newcommand*{\homotopyGroup}[2]%
{\ensuremath{\pi_{#1}\of{#2}}}
\newcommand*{\fundamentalGroup}[1]%
{\ensuremath{\homotopyGroup{1}{#1}}}

\newcommand*{\normalizer}[2]{\ensuremath{N_{#2}#1}}

\newcommand*{\centralizer}[2]{\ensuremath{Z_{#2}#1}}

\newcommand*{\defeq}{\mathrel{\vcenter{\baselineskip0.5ex \lineskiplimit0pt
                     \hbox{\scriptsize.}\hbox{\scriptsize.}}}%
                     =}

\newcommand*{\Lie}[1]{\ensuremath{\mathfrak{\lowercase{#1}}}}
\newcommand*{\MakeLie}[1]{\expandafter\def\csname Lie#1\endcsname{\Lie{#1}}}

\newcommand*{\Kaehler}{K\"ahler\xspace}
\newcommand*{\Aut}[1]{\ensuremath{\operatorname{Aut}\of{#1}}}

\def\lst{A,B,G,H,N,S,T,Z}
\makeatletter
\@for\i:=\lst\do{\expandafter\MakeLie \i}
\makeatother

\newcommand*{\dimC}[1]{\ensuremath{\dim_{\C{}} \! #1}}

\newcommand*{\alb}[1]{\ensuremath{\mathscr{A}_{#1}}}

\begin{document}
\begin{abstract}   
We prove  that any holomorphic locally homogeneous geometric structure on a complex torus of dimension two, modelled on a complex  homogeneous surface, is translation invariant.
We conjecture  that this result is true in any dimension. In higher dimension, we prove it for $G$ nilpotent. We also prove  that  for any given complex algebraic homogeneous space \((X,G)\),   the translation invariant \((X,G)\)-structures  on tori  form a union of connected components in the deformation space of \((X,G)\)-structures.
\end{abstract}

\maketitle
\tableofcontents

\section{Introduction}

We conjecture that holomorphic locally homogeneous geometric structures on complex tori are translation invariant.
Our motivation is from Ghys \cite{Ghys:1996}: holomorphic  nonsingular foliations of codimension one  on any complex torus admit a subtorus of symmetries of codimension at least one.
Let us briefly recall Ghys's  classification of holomorphic codimension one nonsingular foliations on complex tori. The simplest are those given by the kernel of some holomorphic 1-form 
\(\omega\). Since \(\omega\) is necessarily translation invariant on the complex  torus \(T\), the foliation will be also translation invariant.

Assume now that \(T=\C{n}/ \Lambda\), with \(\Lambda\) a lattice in \(\C{n}\) and there exists a linear form \(\pi \colon \C{n} \to \C{}\) sending \(\Lambda\) to a lattice \(\Lambda'\) in \(\C{}\). Then \(\pi\) descends
to a map \(\pi \colon T \to T'\defeq\C{}/ \Lambda'\). Pick a nonconstant meromorphic function \(u\) on the elliptic curve \(T'\) and consider the meromorphic closed 1-form \(\Omega=\pi^{*}(udz) + \omega\) on 
\(T\). It is easy to see that the foliation given by the kernel of \(\Omega\) extends to all of \(T\) as a nonsingular  holomorphic codimension one foliation. This foliation is not invariant by all translations in \(T\), but only by those which lie the kernel of \(\pi\). They act on \(T\) as a subtorus of symmetries of codimension one.

Ghys's theorem asserts that all nonsingular codimension one holomorphic  foliations on complex tori are constructed in this way. In particular, they are invariant by a subtorus of complex codimension one.
Moreover, for generic complex tori, there are no nonconstant meromorphic functions and, consequently, all holomorphic codimension one foliations on those tori  are translation invariant.

Our aim is to generalize Ghys's result to other holomorphic geometric structures  on complex tori and to   find the smallest amount of  symmetry  those geometric structures can have.
In particular, using a result from~\cite{Dumitrescu:2010b,Dumitrescu:2011} we prove here that on  complex tori with no nonconstant meromorphic functions (algebraic dimension zero), all holomorphic geometric structures are translation invariant.

Notice that  there are \emph{holomorphic Cartan geometries} (see section~\ref{section:cartan geometries} for the precise definition) or   holomorphic \emph{rigid geometric structures in Gromov's sense} (see~\cite{Dumitrescu:2011})  on complex tori which are not translation invariant:  just add any holomorphic affine structure (the standard one for example)  to one of the previous   holomorphic  foliations of Ghys. Then the subtorus of symmetries is of codimension one: all symmetries of the foliation preserve the affine structure.

There are holomorphic rigid geometric structures on complex tori without any symmetries. For example, any projective embedding  of an abelian variety into a complex projective space is a holomorphic rigid geometric structure. There  are  no symmetries: any symmetry preserves the fibers of the map; since the map is an embedding the only symmetry is the identity.

Nevertheless, we conjecture that all holomorphic  \emph{locally homogeneous} geometric structures on tori are translation invariant.  
Locally homogeneous geometric structures naturally arise in the following way.  Start with a holomorphic  rigid geometric structure in Gromov's sense  \(\phi\) on a complex manifold \(M\); think of  a holomorphic affine connection or of a holomorphic projective connection. Assume that the local  holomorphic vector fields preserving  \(\phi\) are transitive on the manifold: they span the holomorphic tangent bundle \(TM\) at each point \(m \in M\).  The rigidity  of \(\phi\)  implies that those local vector fields  form a finite dimensional Lie algebra (associated  to a  connected complex Lie group \(G\)). Under these assumptions, \(M\) is then  locally modelled on a complex \(G\)-homogeneous space \(X\)  and, consequently, we get a holomorphic  locally homogeneous geometric structure on $M$  locally modelled on \((X,G)\), also called a holomorphic \((X,G)\)-structure  (see the precise definition in section~\ref{section:Notations and main result}).

For various types of complex homogeneous spaces \((X,G)\), we develop some general techniques below to demonstrate  that all holomorphic \((X,G)\)-structures on complex tori are translation invariant. 
We will prove this below for all \((X,G)\) in complex dimension 1 and 2, by running through the classification of complex homogeneous surfaces following \cite{McKay:2012,McKay/Pokrovskiy:2010,McKay:2014}.
We also prove it for \(G\) nilpotent.

We further conjecture that holomorphic locally homogeneous geometric structures on compact quotients \(P/\pi\), with  \(\pi\) a discrete cocompact subgroup of a complex Lie group \(P\), lift to right invariant geometric structures on \(P\). This will be   proved here only  for those quotients which are of algebraic dimension zero. For \(P=SL(2,\C{})\) this  generalizes a result of Ghys about holomorphic tensors on \(SL(2,\C{})/\pi\)~\cite{Ghys:1995}.

Notice also that these results do not hold in the real analytic category.  For example, there are real analytic flat affine structures  on two dimensional real  tori  (constructed and classified by Nagano and Yagi) which are not translation invariant \cite{Benoist:1999}.

\section{Notation and main result}\label{section:Notations and main result}

A complex homogeneous space \((X,G)\) is a connected complex Lie group \(G\) acting transitively, effectively and holomorphically on a complex manifold \(X\).
A \emph{holomorphic \((X,G)\)-structure} (also known as a \emph{holomorphic locally homogeneous structure modelled on \((X,G)\)}) on a complex  manifold \(M\) is a maximal collection of local biholomorphisms of open subsets of \(M\) to open subsets of \(X\) (the \emph{charts} of the structure), so that any two charts differ by action of an element of \(G\), and every point of \(M\) lies in the domain of a chart.
Every holomorphic \((X,G)\)-structure on a complex manifold \(M\) has a \emph{developing map} \(\delta \colon \tilde{M} \to X\), a local biholomorphism on the universal covering space of \(M\), so that the composition of \(\delta\) with any local section of \(\tilde{M} \to M\) is a  chart of the structure. There is a unique developing map, up to post composition with elements in \(G\).
Any developing map is equivariant for a unique group morphism \(h \colon \fundamentalGroup{M} \to G\), the \emph{holonomy morphism} of the developing map. 
The reader can consult \cite{Goldman:2010} as a reference on locally homogeneous structures.

Throughout this paper, we take an \((X,G)\)-structure on a complex torus \(T=V/\Lambda\),  the quotient of \(V=\C{n}\)  with a lattice \(\Lambda \subset V\),  with holonomy morphism \(h \colon \Lambda \to G\) and developing map \(\delta \colon V \to X\).
Let \(x_0 \defeq \delta(0)\). 
Denote by \(H = G^{x_0}\subset G\) the stabilizer of the  point \(x_0 \in X\) and let \(n=\dimC X\).

Among interesting \((X,G)\)-structures on tori we note:
\[
\begin{array}{@{}rll@{}} \toprule
X & G & \text{\((X,G)\)-structure} \\ \midrule
\C{n} & \GL{n,\C{}} \ltimes \C{n} & \text{complex affine structure} \\
\mathbb{P}^n & PSL(n+1, \C{}) & \text{complex projective structure} \\
\C{n} & O(n, \C{}) \ltimes \C{n} & \text{flat holomorphic Riemannian metric} \\
\pr{\sum_i z_i^2=0} \subset \mathbb{P}^{n+1} & PO(n+2, \mathbb{C}) & \text{flat holomorphic conformal structure} \\ \bottomrule
\end{array}
\]

The main theorem of our article is the following:
\begin{theorem}    
Suppose that \((X,G)\) is a complex homogeneous curve or surface.
Then every holomorphic \((X,G)\)-structure on any complex torus is translation invariant.
\end{theorem}

For tori $T$ of higher dimension, we prove that the translation invariant  \((X,G)\)-structures on $T$ form a union of connected components in the deformation space of  \((X,G)\)-structures (see Theorem \ref{theorem:translation invariant.deformation}).

\section{The symmetry group}\label{section:symmetries}

Suppose that \(M\) is a manifold with an \((X,G)\)-structure with developing map \(\delta \colon \tilde{M} \to X\) and holonomy morphism \(h \colon \fundamentalGroup{M} \to G\).
Let \(Z_{\tilde{M}}\) be the group of all pairs \((f,g)\) so that \(f \colon \tilde{M} \to \tilde{M}\) is a diffeomorphism equivariant for some automorphism \(a \in \Aut{\fundamentalGroup{M}}\), i.e. \(f \circ \gamma = a(\gamma) \circ f\) with quotient \(\bar{f} \colon M \to M\) and \(g \in G\) and \(\delta \circ f = g \delta\).
Let \(f_* \defeq a\). We can see the fundamental group as a  discrete normal subgroup of \(Z_{\tilde{M}}\)  through the map  \(\gamma \in \fundamentalGroup{M} \mapsto \pr{\gamma,h\of{\gamma}} \in Z_{\tilde{M}}\).

The automorphism group of the \((X,G)\)-structure is the quotient 
\[Z_M \defeq Z_{\tilde{M}}/\fundamentalGroup{M}.\] 
Thus, a diffeomorphism of \(M\) is an automorphism just when, lifted to \(\tilde{M}\), it reads through the developing map \(\delta\)  as an element of \(G\).

We extend the holonomy morphism from  \(\fundamentalGroup{M}\)  to \(Z_{\tilde{M}}\)  by  defining  \(h \colon Z_{\tilde{M}} \to G\), \(h(f,g) \defeq g\).

Let \(Z_X \subset G\) be the image of this extended \(h\).  
The universal covering map \(\tilde{M} \to M\) is equivariant for \(Z_{\tilde{M}} \to Z_M\), while the developing map \(\delta \colon \tilde{M} \to X\) is equivariant for \(Z_{\tilde{M}} \to Z_X\).

The notation \(Z_X\) is explained by the following lemma.
\begin{lemma} \label{lemma:centralizer}
The identity component \(Z_X^0 \subset Z_X\) is the identity component of the centralizer \(\centralizer{h\of{\pi}}{G}\) in \(G\) of the image of \(\pi \defeq \fundamentalGroup{M}\).
\end{lemma}
\begin{proof}
Let \(Z \defeq \centralizer{h\of{\pi}}{G}\).
For any \((f,g)\) in the identity component \(Z_{\tilde{M}}^0\) of \(Z_{\tilde{M}}\), \(f_*\) is the identity. Since the action of \(Z_{\tilde{M}}^0\) on \(\tilde{M} \) commutes with the action of \( \pi_1(M) \) and \(\delta\) is equivariant, \(Z_X^0=h(Z_{\tilde{M}}^0)\) commutes with \(h(\pi)\). Hence,  \(Z_X^0\) lies in the centralizer \(Z\).

Conversely, every vector field \(z \) in the Lie algebra of \(Z\) is a complete vector field on \(X\) commuting with \(h\of{\pi}\). 
Pull back  \(z\) by \(\delta\) to become a  complete vector field on \(\tilde{M}\), invariant  under the action of  \(\pi\) by deck transformations, i.e. the vector field descends to \(M\).  The corresponding flow is an one-parameter subgroup  in  \(Z_{\tilde{M}}^0\) whose image under \(h\) is the flow of \(z\). 
Therefore the identity component of \(Z\) lies in  \(Z_X^0=h(Z_{\tilde{M}}^0)\).
\end{proof}


For a complex homogeneous space \((X,G)\), let \(\bar{G}\) be the closure of \(G\) in the biholomorphism group of \(X\), in the topology of uniform convergence on compact sets.
A complex homogeneous space \((X,G)\) is \emph{full} if \(G=\bar{G}\).
If \(\bar{G}\) is a complex Lie group acting holomorphically on \(X\), then the complex homogeneous space \(\pr{X,\bar{G}}\) is the \emph{fulfillment} of \((X,G)\).
It turns out that every complex homogeneous space of complex dimension 1 or 2 has a fulfillment, i.e. \(\bar{G}\) is a finite dimensional complex Lie group acting holomorphically on \(X\).
It does not appear to be known if there are complex homogeneous spaces which do not have a fulfillment.
A locally homogeneous structure is \emph{full} if its model is full.

Take complex numbers \(\lambda_1, \lambda_2, \dots, \lambda_n\), algebraically independent, define the diagonal matrix
\[
A=
\begin{pmatrix}
\lambda_1 \\
 & \lambda_2 \\
 & & \ddots \\
 & & & \lambda_n
\end{pmatrix}
\]
and consider the morphism of complex Lie groups
\[
\mu \colon z \in \C{} \mapsto e^{z A} \in \C{*n}.
\]
Take a faithful representation \(\rho \colon \C{*n} \to \GL{N,\C{}}\), and let \(X=\C{N}\) and \(G \defeq \C{N} \rtimes \C{}\) with action \((v,z)x=\rho \circ \mu(z)x+v\).
Then \((X,G)\) is a complex homogeneous space with fulfillment \(\pr{X,\bar{G}}\) where \(\bar{G}=\C{N} \rtimes \C{*n}\).

\begin{lemma}
Every algebraic homogeneous space \((X,G)\) is full.
\end{lemma}
\begin{proof}
Let \(X^{(k)}\) be the set of all \(k\)-jets of local holomorphic coordinate charts on \(X\) (also called the $k$-frame bundle of $X$). The $G$-action on $X$ lifts to a holomorphic $G$-action on 
\(X^{(k)}\).
For sufficiently large \(k\), every orbit of \(G\) on \(X^{(k)}\) is a holomorphic immersion of \(G\) as a complex submanifold of \(X^{(k)}\) \cite{Baouendi/Rothschild/Winkelmann/Zaitsev:2004} p. 4, Thm. 2.4. In other words this result asserts that  the $G$-action is free on  \(X^{(k)}\) (for  $k$ large enough) as soon as  the stabilizer of any  point  in $X$ has finitely many connected components (which is always  true for algebraic actions).

Moreover,  the  \(G\)-action being supposed  algebraic,  all $G$-orbits of minimal dimension  in \(X^{(k)}\)  are closed in the algebraic Zariski topology \cite{Springer:2009} p. 28 Lemma 2.3.3 (ii). Since  the $G$-action is free on  \(X^{(k)}\), all orbits have the same dimension and they are all closed in the algebraic Zariski topology of \(X^{(k)}\).

Therefore the orbits of the  \(G\)-action on  \(X^{(k)}\)  define an algebraic  foliation  on  \(X^{(k)}\) such that each leaf is an embedded algebraic subvariety biholomorphic to  \(G\). Recall now that the Lie group \(G\)
has its Maurer-Cartan 1-form with values in its Lie algebra (defining a holomorphic parallelization of the holomorphic  tangent bundle of \(G\)) and elements of \(G\) are precisely those biholomorphisms  of \(G\) preserving the Maurer-Cartan 1-form.

Hence the cotangent sheaf of our algebraic  foliation on  \(X^{(k)}\)  admits a global holomorphic section with values in the Lie algebra of   \(G\) which coincides on each  leaf with the Maurer-Cartan 1-form.
Elements in  \(G\) are  precisely those  biholomorphisms  of \(X\) which (when lifted to \(X^{(k)}\))   preserve a leaf (and hence all leafs) of this foliation and its Maurer-Cartan 1-form. Consequently, \(G\) is a closed subgroup in the biholomorphism group of \(X\), for  the topology of uniform convergence on compact sets.
\end{proof}


%
%
%


The elements of \(Z_M\) are precisely the diffeomorphisms of \(M\) that lift to elements of \(Z_{\tilde{M}}\).
If the model \((X,G)\) is full then 
the groups \(Z_{\tilde{M}}, Z_X\) and \(Z_M\) are closed subgroups of the appropriate biholomorphism groups, as limits of symmetries are symmetries.
Even if the model is not full, \(Z_X^0\) is a closed subgroup of the biholomorphism group by lemma~\vref{lemma:centralizer}.
Denote by \(\LieZ\) the Lie algebra of \(Z_X \subset G\), which is also the Lie algebra of \(Z_M\) and of \(Z_{\tilde{M}}\).

For  a complex torus \(T=V/ \Lambda \), the previous notations become  \(Z_M =Z_T\)  and \(Z_{\tilde{M}}=Z_V\). Every element \(\gamma=(f,g) \in Z_V\) has \(f \colon V \to V\) dropping to an automorphism of the torus, so \(f(v)=b+av\) where \(b \in V\) and \(a=f_* \in \GL{V}\) and \(a \Lambda=\Lambda\).
 As above, the holonomy morphism  \(h\) extends from the fundamental group \(\Lambda\)  to a complex Lie group morphism \(h \colon Z_V \to Z_X\) so that \(\delta(\gamma v)=h(\gamma)\delta(v)\), for all \(v \in V\) and \(\gamma \in Z_V\).

The following lemma characterizes those \((X,G)\)-structures on complex tori which are translation invariant.

\begin{lemma}\label{lemma:algebraic.description}
For any complex homogeneous space \((X,G)\) and any holomorphic \((X,G)\)-structure on a complex torus \(T\), with notation as above, the identity component \(Z^0_X\) of \(Z_X \) is an abelian group, acting locally freely on the image of the developing map.
The Lie algebra \(\LieZ \subset V\) is a complex linear subspace acting on \(V\) by translations.
The subgroup \(Z_X \subset G\) is a closed complex subgroup and has a finite index subgroup lying inside \(\centralizer{h\of{\Lambda}}{G}\).
If \((X,G)\) is full then the identity component of the Lie group \(Z_T\) is  a subtorus \(Z_T^0 \subset T\) covered by  \(Z_V^0=\LieZ\).
The following are equivalent:
\begin{enumerate}
\item
The \((X,G)\)-structure on \(T\) is translation invariant.
\item
\(\dimC \LieZ=\dimC T\).
\item
The morphism \(h\) extends to a Lie group morphism \(h \colon V \to G\) for which the developing map is equivariant.
\item
There is a morphism of complex Lie groups \(h \colon V \to G\) with image transverse to \(H\) so that the local biholomorphism \(\delta \colon v \in V \mapsto h(v)x_0 \in X\) is the developing map.
\end{enumerate}
\end{lemma}
\begin{proof}
Every holomorphic vector field on \(T\) is a translation, and translations commute, so \(Z_T^0\) is a complex abelian subgroup in (the translation group) \(T\).
If \((X,G)\) is full, then \(Z_T\) is a closed complex subgroup of \(G\).
No vector field on \(T\) can vanish at a point without vanishing everywhere: \(\LieZ \cap \LieH = 0\), with \(\LieH\) the Lie algebra of the stabilizer \(H \) of the point \(x_0\) in the image of the developing map. 
Since \(\LieZ \subset V\) is a complex linear subspace, \(\dimC{\LieZ} \le \dimC{V}\),  with equality exactly  when \(\LieZ=V\) and then the \((X,G)\)-structure is translation invariant. 
The map \((f,g) \in Z_V \mapsto f'(0) \in \GL{V}\) has finite image, because \(f'(0)\Lambda=\Lambda\).
So the kernel of this map is a finite index subgroup of \(Z_V\), consisting of those pairs \((f,g) \in Z_V\) for which \(f\) is a translation, and therefore commutes with the translations \(\Lambda\), and so \(g\) commutes with \(h(\Lambda)\).
So this finite index kernel maps to \(\centralizer{h(\Lambda)}{G}\).
But it maps to a finite index subgroup of \(Z_X\).
So, using Lemma \ref{lemma:centralizer}, \(Z_X\) is a finite extension of an open subgroup of \(\centralizer{h(\Lambda)}{G}\).
Since \(\centralizer{h(\Lambda)}{G}\) is a closed complex subgroup, its components are closed in the complex analytic Zariski topology, and disjoint.
Therefore the open subgroup is also a closed analytic subvariety, and so \(Z_X\) is also a closed analytic subvariety.
\end{proof}

Hence, to prove translation invariance of  an \((X,G)\)-structure on a complex torus \(T=\C{n}/\Lambda\), one needs to  prove that the centralizer of \(h(\Lambda)\) in \(G\) is of complex dimension \(n\).
We will see in section \ref{section:finite.holonomy} that, at least for \((X,G)\) algebraic, this centralizer is always of positive dimension.

\begin{example}
If we have a translation invariant \((X,G)\)-structure on \(T=\C{n}/\Lambda\), the \emph{same} holonomy morphism \(h\) and developing map \(\delta\) defines a translation invariant \((X,G)\)-structure on \emph{any} complex torus \(T'=\C{n}/\Lambda'\) of the same dimension.
\end{example}

\begin{example} 
Take \(X=\C{2}\) and \(G\) the complex special affine group.
The generic 1-dimensional subgroup of \(G\) has centralizer also 1-dimensional, so it is thus far possible that \(Z_X\) is one dimensional.
We need something more to decide translation invariance.
\end{example}

\begin{example}\label{example:B.beta.1}
Pick a positive integer \(k\).
Let \(X\defeq \C{2}\) and let \(G\) be the set of pairs \((t,f)\) for \(t \in \C{}\) and \(f\) a complex-coefficient polynomial in one variable of degree at most \(k\), with multiplication 
\[
\pr{t_0,f_0}\pr{t_1,f_1}=\pr{t_0+t_1,f_0(u)+f_1\of{u-t_0}}
\]
and action on \(X\)
\[
\pr{t,f}\pr{u,v}=\pr{u+t,v+f\of{u+t}}.
\]
As we vary \(k\) we obtain all of the complex algebraic homogeneous surfaces of class \(B\beta{1}\), in Sophus Lie's notation \cite{McKay:2014}.
One checks easily that all 1-dimensional subgroups of \(G\) have 2-dimensional centralizer, so our group \(Z_X\) must have dimension at least 2.
Therefore for this particular \((X,G)\), every \((X,G)\)-structure on any complex 2-torus is translation invariant.
\end{example}

\section{Finite holonomy} \label{section:finite.holonomy}

Pick a complex vector space \(V\) and lattice \(\Lambda \subset V\). 
Let \(X\defeq G \defeq V/\Lambda'\) for some (possibly different) lattice \(\Lambda' \subset V\) and let \(\delta \colon z \in V \mapsto z+\Lambda' \in X\), so that
\[
h\of{\Lambda}=\Lambda/\pr{\Lambda \cap \Lambda'}.
\]
The tori \(V/\Lambda\) and \(V/\Lambda'\) are isogenous just when \(h(\Lambda) \subset V/\Lambda'\) is a finite set, and then the Zariski closure is just \(h\of{\Lambda}\), finite.
Since \(G\) is abelian (and connected), lemma~\vref{lemma:centralizer} implies \(Z_X=G\), so the \((X,G)\)-structure with developing map \(\delta\) and holonomy morphism \(h\) is translation invariant.

\begin{lemma} \label{lemma:translation invariant}
A holomorphic \((X,G)\)-structure on a complex torus \(V/\Lambda\) has finite holonomy group just when it is constructed as above. In particular, it is translation invariant.
\end{lemma}
\begin{proof}
If \(h(\Lambda)\) is finite, then we can lift to a finite covering of \(T\) to arrange that \(h(\Lambda)\) is trivial, and then, by lemma~\ref{lemma:centralizer}, \(Z_X=G\). In particular, the \((X,G)\)-structure is translation invariant. Lemma  \vref{lemma:algebraic.description} implies that 
\(G\) is an abelian group acting locally freely on \(X\), meaning that \(X\) is a quotient of \(G\) by a discrete subgroup.

The developing map descends
to \(T \). Its image is open and closed in \(X\), so the developing map is onto: it is a finite cover of  \(X\), by \(T\).  Consequently, \(X\) is  also a complex torus. Since \(G\) is connected and acts transitively and effectively on the complex torus  \(X\), \(G\) is the  translation group  \(X\). 
\end{proof}

\begin{corollary}\label{corollary:positive.dim.symmetries}
Suppose that \((X,G)\) is a complex algebraic homogeneous space.
Any holomorphic \((X,G)\)-structure on a complex torus has positive dimensional symmetry group.
\end{corollary}
\begin{proof}
Assume, by contradiction, that  \(\dimC{\LieZ}=0\).  Then \(\centralizer{h(\Lambda}{G})\) is finite, being a discrete  algebraic subgroup of \(G\). 
But \(h(\Lambda) \subset \centralizer{h(\Lambda)}{G}\), since \(\Lambda\) is abelian.
Therefore \(h(\Lambda)\) is finite.
Lemma~\ref{lemma:translation invariant} implies that \(Z_X=G\): a contradiction.
\end{proof}

\begin{corollary}\label{corollary:curves}
If a smooth compact complex curve has genus at most 1, then every holomorphic locally homogeneous structure on the curve is homogeneous.
If the curve has genus more than 1, then no holomorphic locally homogeneous structure on the curve is homogeneous.
\end{corollary}
\begin{proof}
Every complex homogeneous curve \((X,G)\) is algebraic \cite{McKay:2011b} p. 14 Theorem 2.
By Corollary~\ref{corollary:positive.dim.symmetries}, any holomorphic \((X,G)\)-structure on any elliptic curve has positive dimensional symmetry group, with identity component consisting of translations, so is translation invariant.
Any locally homogeneous structure on a simply connected compact manifold is identified with a cover of the model by the developing map, so there is only one locally homogeneous structure on \(\Proj{1}\).
Higher genus Riemann surfaces have no nonzero holomorphic vector fields.
\end{proof}

\section{Discrete stabilizer}

\begin{lemma}\label{lemma:X.is.G}
Suppose that \((X,G)\) is a complex homogeneous space and that \(\dimC X = \dimC G\).
Then every holomorphic \((X,G)\)-structure on any complex torus is translation invariant.
\end{lemma}
\begin{proof} Here \(X=G/H \), with \(H \) a discrete subgroup in \(G \).
Lift the developing map uniquely to a map to \(G\), so that \(\delta(0)=1\), and then 
\[
\delta(x+\lambda)\pr{h(\lambda)\delta(x)}^{-1} \in H,
\]
i.e.
\[
\delta(x+\lambda)\delta(x)^{-1}h(\lambda)^{-1} \in H
\]
is constant, because the stabilizer \(H\) has dimension zero.
Plug in \(x=0\) to find \(\delta(x+\lambda) = \delta(\lambda)\delta(x)\), i.e. we can arrange that \(H=\curly{1}\) and \(h=\left.\delta\right|_{\Lambda}\).

Consider the universal covering group \(\tilde{G} \to G\).
The developing map \(\delta \colon V \to G\) lifts uniquely to a map \(\tilde{\delta} \colon V \to \tilde{G}\) so that \(\tilde{\delta}(0)=1\).
By the same argument, \(\tilde{\delta}(x+\lambda)=\tilde{\delta}(\lambda)\tilde{\delta}(x)\) for all \(x \in V\) and \(\lambda \in \Lambda\), i.e. \(\tilde{\delta}\) is a developing map for a \(\pr{\tilde{G},\tilde{G}}\)-structure.
So without loss of generality, we can assume that \(X=G\) and that \(G\) is simply connected.

Consider the map
\[
\Delta \colon (x,y) \in V \times V \mapsto \delta(x)^{-1}\delta(y)^{-1}\delta(x+y) \in G.
\]
Clearly \(\Delta(x,y+\lambda)=\Delta(x,y)\) if \(\lambda \in \Lambda\).
So \(\Delta \colon V \times T \to G\) is holomorphic.
Fixing \(x\), \(y \mapsto \Delta(x,y) \in G\) is a holomorphic map from a complex torus to a simply connected complex Lie group, and therefore is constant \cite[p. 139, theorem 1]{Matsushima/Morimoto:1960}.
So \(\Delta(x,y)=\Delta(x,0)=1\) for all \(x,y\), i.e. \(\delta \colon V \to G\) is a holomorphic Lie group morphism, hence, by Lemma \ref {lemma:algebraic.description} point  (3),  a translation invariant \((X,G)\)-structure.
\end{proof}

\begin{example}
From the classification of complex homogeneous surfaces \((X,G)\) \cite{McKay:2014}, the surfaces \(D1, D1_1, \dots, D1_5\) and \(D2, D2_1, \dots, D2_{14}\) are the smooth quotients of 2-dimensional complex Lie groups by various discrete groups, i.e. they are precisely the complex homogeneous surfaces \((X,G)\) with \(2=\dimC{X}=\dimC{G}\).
By lemma~\vref{lemma:X.is.G}, all \((X,G)\)-structures on complex tori, with \((X,G)\) among the surfaces \(D1, D1_1, \dots, D1_5\) and \(D2, D2_1, \dots, D2_{14}\), are translation invariant.
\end{example}

\section{Enlarging the model and its symmetry group}

A \emph{morphism} \((X,G) \to \pr{X',G'}\) of complex homogeneous spaces is a holomorphic map \(X \to X'\) equivariant for a holomorphic group morphism \(G \to G'\).
If moreover \(X \to X'\) is a local biholomorphism, then every \((X,G)\)-structure induces an \(\pr{X',G'}\)-structure by composing the developing map with \(X \to X'\) and the holonomy morphism with \(G \to G'\), and any \(\pr{X',G'}\)-structure is induced by at most one \((X,G)\)-structure.

Lemma~\vref{lemma:algebraic.description} together with  lemma~\vref{lemma:X.is.G} lead to the following corollary:

\begin{corollary}
For any complex homogeneous space \((X,G)\), an \((X,G)\)-structure on a complex torus is translation invariant just when it is induced by an \(\pr{X_0,G_0}\)-structure, where \(G_0 \subset G\) is a connected complex subgroup acting transitively and locally freely on an open set \(X_0 \subset X\) and if this occurs then \(G_0\) is abelian.
\end{corollary}

In the statement above $X_0$ is the image of the developing map of the \((X,G)\)-structure, seen as a homogeneous space of the Lie group $G_0=Z^0_X$.

\begin{proposition}\label{proposition:enlarge}
Suppose that \((X,G) \to \pr{X',G'}\) is a morphism of complex homogeneous spaces for which \(X \to X'\) is a local biholomorphism and \(G \to G'\) has closed image \(\bar{G} \subset G'\).
Suppose that there is no positive dimensional compact complex torus in \(G'/\bar{G}\) acted on transitively by a subgroup of \(G'\).
For example, there is no such torus when \(G'\) is linear algebraic and \(G \to G'\) is a morphism of algebraic groups.
Every translation invariant \(\pr{X',G'}\)-structure on any complex torus with holonomy contained in \(\bar{G}\) is induced by a unique \((X,G)\)-structure, which is also translation invariant.
\end{proposition}
\begin{proof}
Denote the developing map and holonomy morphism of the \(\pr{X',G'}\)-structure by \(\delta'\) and \(h'\).
Since the structure is translation invariant, extend \(h'\) to a complex Lie group morphism \(h' \colon V \to G'\) so that \(\delta' \colon V \to X\) is just \(\delta'(v)=h'(v)x_0'\).
Denote the morphism \(G \to G'\) as \(\rho \colon G \to G'\).
The holonomy morphism \(h'\) descends to a complex Lie group morphism
\(
T \to G'\!/\bar{G}
\).

By hypothesis, this is constant: \(h'\) has image in \(\bar{G}=\rho(G)\).
The developing map is \(\delta'(v)=h'(v)x_0'\) so has image in the image of \(X \to X'\).
On that image, \(X \to X'\) is a covering map, by \(G \to G'\) equivariance, so \(\delta' \colon \pr{V,0} \to \pr{X',x'_0}\) lifts to a unique local biholomorphism \(\delta \colon \pr{V,0} \to \pr{X,x_0}\).
Similarly, the morphism \(h' \colon V \to \rho(G)\) lifts uniquely to a morphism \(h \colon V \to G\).
By analytic continuation \(\delta(v)=h(v)x_0\) for all \(v \in V\), so that the \((X,G)\)-structure is translation invariant.

Suppose that \(G'\) is linear algebraic and \(\rho \colon G \to G'\) is a morphism of algebraic groups.
The quotient of a linear algebraic group by a Zariski closed normal subgroup is linear algebraic \cite{Borel:1991} p.93 theorem 5.6, so \(Z_{X'}/\rho\of{Z_X}\) is a linear algebraic group and therefore contains no complex torus subgroup.
\end{proof}

\begin{example}
 If \(G\) is the universal covering space of the group of complex affine transformations of \(\C{}\), and \(X=G\) acted on by left translation, then the center of \(G\) consists in the deck transformations over the complex affine group.
 The surface \((X,G)\) is not algebraic, but the quotient \(\pr{X',G'}\) by the center is algebraic; any \((X,G)\)-structure induces and arises uniquely from an \(\pr{X',G'}\)-structure.
\end{example}

\begin{example}
The classification of the complex homogeneous surfaces \cite{McKay:2014} yields unique morphisms
\(A2\to A1\), 
\(A3\to A1\), 
\(A3\to A2\), 
\(B\beta{1} \to B\beta{2}\),
\(B\beta{1}B0 \to B\beta{2}'\), 
\(B\gamma{1} \to B\delta{4}\),
\(B\gamma{2} \to B\delta{4}\),
\(B\gamma{3} \to B\delta{4}\),
\(B\gamma{4} \to B\delta{4}\),
\(B\delta{1} \to B\delta{2}\),
\(B\delta{1}' \to B\delta{2}'\),
\(B\delta{3} \to B\delta{4}\),
\(C2 \to C7\),
\(C2' \to C5'\),
\(C3 \to C7\),
\(C5 \to C7\),
\(C6 \to C7\),
\(C8 \to A1\),
\(D1 \to A1\),
\(D1_1 \to C7\),
\(D1_2 \to C7\),
\(D1_3 \to C5'\),
\(D1_4 \to C5'\),
\(D2 \to A1\),
\(D3 \to A1\).
For each of these morphisms \((X,G) \to \pr{X',G'}\),  \(G'/G\) contains no homogenenous complex torus.
Below we will prove that all \(\pr{X',G'}\)-structures on complex tori are translation invariant.

It follows that all \((X,G)\)-structures on complex tori are translation invariant, for each of these morphisms \((X,G) \to \pr{X',G'}\).
This reduces the proof of translation invariance of \((X,G)\)-structures on tori for most of the transcendental surfaces \((X,G)\) to the same problem for algebraic surfaces \(\pr{X',G'}\).
\end{example}

\section{Normalizer chain of the holonomy}

Continue with our notation as in section~\ref{section:Notations and main result}: \((X,G)\) is a complex homogeneous space, \(x_0 \in X\) some point, \(H \subset G\) the stabilizer of \(x_0\), \(T=V/\Lambda\) is a complex torus, \(\delta \colon \pr{V,0} \to \pr{X,x_0}\) is the developing map and \(h \colon \Lambda \to G\) the holonomy morphism for an \((X,G)\)-structure on \(T\).
Extend \(h\) as above to a morphism of Lie groups \(h \colon Z_V \to Z_X\).
Let \(S_{-1}\defeq \curly{1}\), \(S_0\defeq Z^0_X\) and let \(S_{i+1}\defeq \pr{\normalizer{S_i}{G}}^0\) with Lie algebras \(\LieS_i\). 
Recall that \(\pr{\normalizer{S_i}{G}}^0\) is the identity component of the normalizer of \(S_{i}\) in \(G\). 
Call the sequence \(S_{-1} \trianglelefteq S_0 \trianglelefteq \dots\) the \emph{normalizer chain} of the structure.
Since \(S_0\) is \(\Ad{h(\Lambda)}\)-invariant, so are all of the \(S_i\).

\begin{lemma}\label{lemma:normalizing.sequence}
Consider an   \((X,G)\)-structure on a complex torus \(T\).
The groups \(S_0 \trianglelefteq S_1 \trianglelefteq \dots \) in the normalizer chain of that structure are solvable connected complex Lie groups with abelian quotients \(S_{i+1}/S_i\).
Each of these groups acts locally freely on the image of the developing map of the \((X,G)\)-structure.
\end{lemma}
\begin{proof}
Lemma~\vref{lemma:algebraic.description} proves that \(S_0=Z^0_X\) is abelian and acts locally freely at every point in \(\delta(V)\).
Each element of \(\LieS_1 \subset \LieG\) is a vector field on \(X\), whose flow preserves the Lie subalgebra \(\LieS_0\).
Such a vector field pulls back via the local biholomorphism \(\delta\) to a vector field on \(V\), whose flow preserves the translations \(\LieS_0 = \LieZ \subset V\).
The \(\LieS_1\) vector fields on \(V\) locally descend to \(T\), but globally they only do so modulo transformations of \(\Lambda\), which add elements of \(\LieS_0\).
The Lie brackets of the \(\LieS_1\) vector fields are only defined on \(T\) modulo the \(\LieS_0\) vector fields.
The part of the bracket lying in the quotient \(\LieS_1/\LieS_0\) is a holomorphic map \(T \to \Lm{2}{\LieS_1/\LieS_0}^* \otimes \pr{\LieS_1/\LieS_0}\), so constant.
This constant gives the structure constants  of the Lie algebra \(\LieS_1/\LieS_0\). 
The normal bundle of the  foliation inherited  in \(T\) by the \(S_0\)-action admits an \(S_0\)-invariant integrable subbundle with fiber isomorphic to \(\LieS_1/\LieS_0\): it is a partial transverse structure to the foliation modelled on \(S_1/S_0 \).

Split the tangent bundle of \(T\) by some linear splitting \(V=\LieS_0 \oplus \LieS_0^{\perp}\).
Since \(S_0\) acts by translations, this splitting is preserved.
The normal bundle to the foliation sits inside the tangent bundle of \(T\), and every vector field from \(\LieS_1/\LieS_0\) is represented as a vector field on the torus, hence a translation field.
The brackets of these vector fields on the torus agree, modulo the constant translations in \(\LieS_0\), with those of \(\LieS_1/\LieS_0\).
But Lie brackets of holomorphic vector fields on the torus are trivial, so \(\LieS_1/\LieS_0\) is abelian.

Each of the vector fields arising from this splitting is translation invariant, so has vanishing normal component at a point in \(T\) just when its normal component vanishes everywhere on \(T\), i.e. just when the associated element in \(\LieS_1\) belongs to \(\LieS_0\).
An element of \(\LieS_1\) pulls back by the developing map to \(V\) to agree with an element of \(\LieS_0\) at some point just when they agree at every point of \(V\), and so they agree at every point of \(\delta(V)\).
In other words, \(\LieS_1\) acts locally freely on \(\delta(V)\).
The same argument holds by induction for the successive subgroups \(\LieS_i \subset \LieS_{i+1}\).
\end{proof}

\begin{proposition}
Suppose that \((X,G)\) is a complex homogeneous space and that \(T\) is a complex torus with a holomorphic \((X,G)\)-structure.
As above, let \(S_{-1} \subset S_0 \subset S_1 \subset \dots \subset G\) be the normalizer chain of the holonomy morphism and let \(S=\bigcup_i S_i\), i.e. \(S\) be the terminal subgroup in the chain of connected complex subgroups.
Then \(\dimC{S} \le \dimC{X}\) with equality if and only if the \((X,G)\)-structure is translation invariant.
\end{proposition}
\begin{proof}
By lemma~\vref{lemma:normalizing.sequence}, \(S\) acts locally freely on \(\delta(V)\). Consequently, \(\dimC{S} \le \dimC{X}\).  If \(\dimC{S} < \dimC{X}\), then \(\dimC{S_0}=\dimC{Z_X^0} < \dimC{X}\) and
thus the \((X,G)\)-structure is not translation invariant.

Assume now that \(\dimC{S} = \dimC{X}\). Replace \(G\) by \(S\) (modulo any elements of \(S\) acting trivially on \(\delta(V)\)) and \(X\) by \(\delta(V)\) to arrange that \(G\) acts on \(X\) locally freely, so \(\dim G=\dim X\) and apply lemma~\vref{lemma:X.is.G}.
\end{proof}

\begin{theorem}\label{theorem:nilpotent}
If  \((X,G)\) is a complex homogeneous space and \(G\) is nilpotent then every holomorphic \((X,G)\)-structure on any complex torus is translation invariant.
\end{theorem}
\begin{proof}
The normalizer chain always increases in dimension until it reaches the dimension of \(G\) \cite{Borel:1991} p. 160.
\end{proof}

\begin{example}
For the surfaces \((X,G)\) in example~\vref{example:B.beta.1}, and even for transcendental \(B\beta{1}\)-surfaces and \(B\beta{2}\)-surfaces \cite{McKay:2014}, \(G\) is nilpotent.
The nilpotent complex homogeneous surfaces \((X,G)\) are \(B\beta{1}\), 
\(B\beta{2}\), 
\(D1\), 
\(D1_1, \dots, D1_5\),
\(D2\), 
\(D2_1, \dots, D2_{14}\) \cite{McKay:2014}.
Therefore for all of these surfaces \((X,G)\), all \((X,G)\)-structures on complex tori are translation invariant.
\end{example}

\section{Algebraic dimension}

If \((X,G)\) is a complex algebraic homogeneous space then any  holomorphic \((X,G)\)-structure is a holomorphic rigid geometric structure in Gromov's sense \cite{Dumitrescu:2001} and also
a  (flat) Cartan geometry  (see the definition in section \ref{section:cartan geometries}).

Recall that the \emph{algebraic dimension} of a complex manifold \(M\) is the transcendence degree of the field of meromorphic functions of \(M\) over the field of complex numbers. A generic torus
has algebraic dimension zero, meaning that all its meromorphic functions are constant~\cite{Ueno:1975}.

\begin{lemma}\label{lemma:symmetries.torus}
The identity component of the symmetry group of any holomorphic geometric structure on a complex torus \(T\) acts as a subtorus  \(T_0\) of dimension at least the algebraic codimension of \(T\) (i.e. \(n-\kappa\), where \(n=\dimC{T}\) and \(\kappa\) is the algebraic dimension of \(T\)).

The quotient of \(T\) by the subtorus \(T_0\) is an abelian variety (which coincides with the algebraic reduction of \(T\) if and only if \(T_0\)  is of complex dimension \(n-\kappa\)).
\end{lemma}

\begin{proof}
The pair of the holomorphic geometric structure and the translation structure (the holomorphic parallelization) of \(T\)  is a holomorphic rigid geometric structure on \(T\). 
The symmetry pseudogroup of any such structure acts transitively on sets of codimension \(\kappa\) \cite{Dumitrescu:2010b,Dumitrescu:2011}.
Therefore near each point there are locally defined holomorphic vector fields preserving both the holomorphic geometric structure and the translation structure (the holomorphic parallelization), acting with orbits of dimension \(\ge \kappa\).
Each of these vector fields preserves the translation structure, so is a translation.
Translations on \(T\)  extend globally, and give global symmetries.
The family of symmetries is Zariski closed in the complex analytic Zariski topology, so forms a subtorus.
\end{proof}

\begin{corollary}  \label{corollary:symmetries.torus} Suppose that \((X,G)\) is a complex algebraic homogeneous space and  \(T\) is a complex torus of algebraic dimension zero. Every holomorphic \((X, G) \)-structure on \(T\) is translation invariant.
\end{corollary}

\begin{proof}  Here the holomorphic geometric structure in the previous proof is the \((X,G)\)-structure. If \(T \) is of algebraic dimension zero, then the subtorus of common symmetries of the \((X,G)\)-structure and of  the translation structure of \(T \) acts transitively. Consequently, the \((X,G)\)-structure is translation invariant.
\end{proof}

The results from  \cite{Dumitrescu:2010b,Dumitrescu:2011} that we used in the proof of lemma~\ref{lemma:symmetries.torus} hold not only for tori, but  for all complex manifolds: any holomorphic rigid geometric structure  (or holomorphic Cartan geometry modelled on an algebraic homogeneous space)  on a complex manifold of algebraic dimension zero is locally homogeneous on a dense open set (away from a nowhere dense analytic subset of positive codimension).
 
With the same method we can then prove the following:

\begin{theorem} \label{theorem:algebraic.dimension.zero} Let \(M \defeq P / \pi\) be a  compact quotient of a complex Lie group \(P\) by a lattice \(\pi\).  If \(M\) is of algebraic dimension zero, then any holomorphic geometric structure \(\phi\) on \(M\) pulls back to a 
translation invariant geometric structure on \(P\). Consequently, for a complex algebraic homogeneous space \((X,G)\), any \((X,G)\)-structure on \(M\) pulls back to \(P\) to a right invariant \((X,G)\)-structure.
\end{theorem}

\begin{proof} We add together the geometric structure \(\phi\) and the holomorphic parallelization of \(TM\) to give a holomorphic rigid  geometric structure \(\phi'\). Then Corollary 2.2  in \cite{Dumitrescu:2011} shows that
\(\phi'\) is locally homogeneous on an open dense set in \(M\), in the sense that the local  holomorphic vector fields preserving both the holomorphic parallelization and \(\phi\) are transitive on an open dense set in \(M\).
But the local  holomorphic vector fields preserving the holomorphic parallelization (which is given by global right invariant vector fields) are those vector fields which are left invariant (they are locally defined on \(M\) and their pull back is globally defined on \(P\)). Hence, all left invariant vector fields on \(P\)  must preserve the pull back of 
\(\phi\); if not \(\phi'\) is not locally homogeneous on any open set. Left invariant vector fields generate right translation: consequently, the pull back of \(\phi\) on \(P\)  is invariant by right translation.

If \(\phi\) is defined by an \((X,G)\)-structure with \((X,G)\) a complex algebraic homogeneous space, then the pull back of the \((X,G)\)-structure to \(P\) is right invariant.
\end{proof}

Theorem~\ref{theorem:algebraic.dimension.zero} generalizes a result of Ghys dealing with  holomorphic tensors on \(SL(2,\C{})/\pi\)~\cite{Ghys:1995}.

\begin{theorem}
Consider  a compact complex manifold \(M\) of complex dimension \(n\), algebraic dimension zero and Albanese dimension \(n\).
Then \(M\) admits a holomorphic rigid geometric structure  (or a holomorphic Cartan geometry  modelled on an algebraic homogeneous space) if and only if \(M\) is a complex torus and the holomorphic rigid geometric structure (the Cartan geometry)  is translation invariant. 
\end{theorem}
\begin{proof}
Since \(M\) is of algebraic dimension zero, it is known that the Albanese map \(M \to \alb{M}\) is surjective, with connected fibers and the Albanese torus  \(\alb{M}\) contains no closed complex hypersurface (i.e. divisor) \cite{Ueno:1975}  lemmas 13.1, 13.3, 13.6.  Here, \(\alb{M}\)  is of the same dimension as \(M\), so the Albanese map \(M \to \alb{M}\) is a modification (see lemma 13.7 in \cite{Ueno:1975}).

Let \(H\) be the locus in \(M \) on which the Albanese map drops rank.
The image \(\alpha(H)\)  of \(H\) through  the Albanese map \(\alpha\) is a nowhere dense analytic subset of \(\alb{M}\)  of complex codimension at least two.

The Albanese map is a biholomorphism between the  open sets \(M \setminus \alpha^{-1}(\alpha(H))\) in \(M\)  and \(\alb{M} \setminus \alpha(H)\) in \(\alb{M}\). This implies that the holomorphic rigid geometric structure of \(M\) drops down to a  holomorphic geometric structure  \(\phi\)  on \(\alb{M} \setminus \alpha(H)\).  Now we put together \(\phi\) and the translation structure (holomorphic parallelization) of \(\alb{M}\) together to form a holomorphic rigid
geometric structure \(\phi'\) on \(\alb{M} \setminus \alpha(H)\). The complex manifold  \(\alb{M} \setminus \alpha(H)\) being of algebraic dimension zero, the geometric structure \(\phi'\) is locally homogeneous on an open dense set in \(\alb{M} \setminus \alpha(H)\) \cite{Dumitrescu:2010b,Dumitrescu:2011}.  The local infinitesimal symmetries are translations, because they preserve the translation structure.  They extend to global  translations on \(\alb{M}\) preserving  \(\phi'\). Consequently, \(\phi'\) is the restriction to \(\alb{M} \setminus \alpha(H)\) of  a translation invariant geometric structure defined on all of  \(\alb{M}\). 

Consider a family of linearly independent translations on \(\alb{M}\). They pull back to commuting holomorphic vector fields  \(\kappa_1, \dots, \kappa_n\) on \(M \setminus \alpha^{-1}(\alpha(H))\) which preserve the initial geometric structure and parallelize  the holomorphic tangent bundle \(TM\) over \(M \setminus \alpha^{-1}(\alpha(H))\). Since they are symmetries of an analytic rigid geometric structure,  they extend to  all of  \(M\) \cite{Amores:1979,Nomizu:1960}.

Pull back a holomorphic volume form by the Albanese map: a  holomorphic section \(\operatorname{vol}\) of the canonical bundle of \(M\) which  vanishes on the branch locus  \(H\).
Plug \(\kappa_1, \dots, \kappa_n\) into the volume form and get the holomorphic function 
\[
\operatorname{vol}\of{\kappa_1, \dots, \kappa_n}
\]
which is constant and nonzero  on \(M \setminus \alpha^{-1}(\alpha(H))\), since it corresponds to a constant nonzero function on the Albanese torus, so constant and nonzero on all of \(M\).

This implies that the holomorphic vector fields \(\kappa_1, \dots, \kappa_n\)  holomorphically parallelize \(TM\) on all of \(M\). Since they commute, \(M\) is a complex torus and \(\kappa_1, \dots, \kappa_n\)
are translation vector fields. The initial geometric structure on \(M\) is translation invariant.
\end{proof}

\section{
\texorpdfstring{Deformation space of \((X,G)\)-structures}{Deformation space of (X,G)-structures}
}

Consider an \((X,G)\)-structure on a manifold \(M\) and the corresponding holonomy morphism \(h : \pi_1(M) \to G\).  The deformation space of \((X,G)\)-structures on \(M\) is the quotient of the space
of \((X,G)\)-structures on \(M\) by the group  of diffeomorphisms of \(M\) isotopic to the identity.

By the Ehresmann-Thurston principle (see, for instance, \cite{Goldman:2010} p. 7), the deformation space of \((X,G)\)-structures on   \(M \)  is locally homeomorphic, through the holonomy map,  to
an open neighborhood of \(h\) in the space of group homomorphisms from  \( \pi_1(M) \)  into \(G \) (modulo the action of  \(G \) on the target \(G\) by  inner automorphisms). 

In other words, any group homomorphism from \( \pi_1(M) \) into \(G \) close to \(h\)  is itself the holonomy morphism of an \((X,G)\)-structure on \(M \) close to the initial one. Also, two close  \((X,G)\)-structures with the same holonomy morphism are each conjugated to the other by an isotopy of \(M\).

For any finitely generated group \(\pi\) and any algebraic group \(G\), \(\Hom{\pi}{G}\) is an algebraic variety (a subvariety of \(G^k\),  if \(\pi\) can be generated by \(k\) elements). If $G$
is a complex Lie group, the space \(\Hom{\pi}{G}\) is a complex  analytic space. A family of \((X,G)\)-structures on $M$ parametrized by a complex reduced space $S$ is called holomorphic if the family of the corresponding holonomy morphisms lifts as a complex analytic map from $S$ to \(\Hom{\pi}{G}\).

Notice that, in our case, the \(G \)-action preserves a complex structure on \(X \) and hence any \((X,G)\)-structure on a manifold \(M \) induces an underlying complex structure on \(M \) (for which the
\((X,G) \)-structure is holomorphic). In particular, when deforming the \((X,G) \)-structure on \(M \),  one also deforms the complex structure on \(M\).

Let us make precise  how the complex structure varies under the  deformation of a holomorphic  \((X,G)\)-structure  on a  complex torus.

\begin{lemma}\label{lemma:perturb}
Consider  a complex homogeneous space \((X,G)\).
Suppose that we have a holomorphic \((X,G)\)-structure on a complex \(n\)-torus \(T=V/\Lambda\), with holonomy morphism \(h \colon \Lambda \to G\).
If \(h_s \in \Hom{\Lambda}{G}\) is a holomorphic family of group morphisms for \(s\)  in some reduced complex space \(S\), with \(h_{s_0}=h\) for some \(s_0 \in S\),  then there is a holomorphic family of \((X,G)\)-structures on a holomorphic family of complex tori \(T_s\) with holonomy morphism \(h_s\), for \(s\) in an open neighborhood of \(s_0\).
\end{lemma}
\begin{proof}
By the Ehresmann--Thurston principle, there is a unique nearby \((X,G)\)-structure on the same underlying real manifold with holonomy morphism \(h_s\).
Since \(G\) preserves a complex structure on \(X\), this \((X,G)\)-structure is holomorphic for a unique complex structure on \(T_s\).
Being a small deformation of a complex torus, \(T_s\) is a complex torus by Theorem 2.1 in  \cite{Catanese}.
\end{proof}

The following result deals with the deformation space of translation invariant \((X,G)\)-structures on complex tori. Notice that the condition of the symmetry group \(Z_X\) being of dimension \(n\) is closed under deformation of \((X,G)\)-structures,  since a limit of \((X,G)\)-structures could have smaller holonomy group (i.e. some generators of \(\Lambda\) landing in some special position), but that would only decrease the collection of conditions that determine the centralizer \(Z_X^0\), so make \(Z_X^0\) larger. Hence   limits of translation invariant \((X,G)\)-structures  are translation invariant, as the centralizer of the holonomy can only increase in dimension.

\begin{theorem} \label{theorem:translation invariant.deformation}  Let \((X,G)\) be a complex algebraic homogeneous space.
If a holomorphic \((X,G)\)-structure on a complex torus \( T= V/\Lambda\) is translation invariant, then so is any deformation of that structure.  Consequently, translation invariant \((X,G)\)-structures form a union of connected components in the deformation space of \((X,G)\)-structures on the (real) manifold \(T \).
\end{theorem}
\begin{proof}
Start with a translation invariant \((X,G)\)-structure on \(T \).
The holonomy morphism extends from \( \Lambda \)  to a complex Lie group morphism \(h \colon V \to G\).
For any other complex torus \(T'=V/\Lambda' \), restrict \(h\) to the period lattice \(\Lambda'\) of that torus and take the same developing map to construct a (translation invariant)  \((X,G)\)-structure on \(T'\).
In particular, we can deform the starting  translation invariant    \((X,G)\)-structure to another (translation invariant) \((X,G)\)-structure on a complex torus of algebraic dimension zero (by choosing a generic
lattice \( \Lambda' \)).  Moreover, by corollary  \vref{corollary:symmetries.torus},    all nearby \((X,G)\)-structures are translation invariant, since  the underlying  complex  structure of the  torus \(T \) remains of  algebraic dimension zero under small perturbation of the complex structure \(T'\).   Hence, the  translation invariant  \((X, G) \)-structures on \(T\)  form  an  open dense set in  our connected component of  the deformation space of \((X,G)\)-structures.  In particular, we proved that the natural map associating to an \((X,G)\)-structure the underlying complex structure on \(T\) is surjective on the Kuranishi space of \(V/\Lambda\).

All of these deformations of  the   \((X,G)\)-structure merely perturb the holonomy morphism  \(h\) through a family of complex Lie group morphisms \(h \colon V \to  G\) and the developing map is
 \(\delta(v)=h(v)x_0\), with \(x_0 \in X\).
 
But, we have seen that  the translation invariant \((X,G)\)-structures always  form a closed set. 
Therefore in that connected component of the deformation space of  \((X,G)\)-structures, all \((X,G)\)-structures are translation invariant.
\end{proof}

Let us give an easy argument implying, for various complex homogeneous surfaces \((X,G)\), that all \((X,G)\)-structures on complex tori of complex dimension two are translation invariant.

\begin{lemma}\label{lemma:Z.nought}
Suppose that \((X,G)\) is a complex algebraic homogeneous space.
If there is a holomorphic \((X,G)\)-structure on a complex torus \(T\), and that structure is not translation invariant, then there is another such structure on another complex torus nearby to a finite covering of \(T\), also not translation invariant, with holonomy morphism having dense  image  in \(Z_X^0\).
Any connected abelian subgroup near enough to \(Z_X^0\) is the identity component of its centralizer and arises as the Zariski closure of the image of the holonomy of a nearby \((X,G)\)-structure on a nearby complex torus.
\end{lemma}
\begin{proof} Since \(Z_X\) is the centralizer of \(h(\Lambda)\) in \(G\), it is an algebraic subgroup in \(G\). Therefore it consists of a finite number of connected components.
After perhaps replacing \(T\) by a finite cover  of \(T\),  we can assume  that \(h(\Lambda) \subset Z_X^0\).
By lemma~\vref{lemma:algebraic.description}, \(Z_X^0\) is abelian. 
Since \(\Lambda\) is free abelian, morphisms \(\Lambda \to Z_X^0\) are precisely arbitrary choices of where to send some generating set of \(\Lambda\).
Since \(\Lambda\) has rank \(2n\) and \(\dimC{Z_X^0} < n\), we can slightly deform the holonomy morphism to have Zariski dense image in \(Z_X^0\).
If we can perturb \(Z_X^0\) slightly to an abelian subgroup with larger centralizer, we can repeat the process. Since we stay in the same connected  component in the deformation space of \((X,G) \)-structures, theorem \vref{theorem:translation invariant.deformation}
implies that none of these \((X,G)\)-structures are translation invariant.
\end{proof}

\begin{example}
Suppose that \(\dimC{X}=2\) and the  Levi decomposition of \(G\) has reductive part with rank 2 or more.
Suppose we have a holomorphic \((X,G)\)-structure on a complex 2-torus.
The generic connected 1-dimensional subgroup of \(G\) is not algebraic, because the characters on a generic element of \(\LieG\) have eigenvalues with irrational ratio.
After perhaps a small perturbation of the \((X,G)\)-structure, \(Z_X^0\) has complex dimension 2 or more: the \((X,G) \)-structure becomes translation invariant. 
Every holomorphic \((X,G)\)-structure of this kind on a complex 2-torus must be  translation invariant (because of  lemma \ref{lemma:Z.nought}).
From the classification of the complex homogeneous surfaces \((X,G)\) \cite{McKay:2014}, this occurs for the complex homogeneous surfaces \(A1\), \(A2\), \(B\beta{2}\), \(B\gamma{4}\), \(B\delta{2}\), \(B\delta{4}\), \(C2\), \(C3\), \(C5\), \(C6\), \(C7\) and \(D1\).
\end{example}

\section{Reductive and parabolic Cartan geometries} \label{section:cartan geometries}

Pick a complex homogeneous space \((X,G)\),  with \(H \subset G\) the stabilizer of a point \(x_0 \in X\), and with the groups \(H \subset G\) having Lie algebras \(\LieH \subset \LieG\).
The space \((X,G)\) is \emph{reductive} if \(H\) is a reductive linear algebraic group, \emph{rational} if \(X\) is compact and birational to projective space and \(G\) is semisimple in adjoint form.

Recall that a \emph{Cartan geometry} (or a \emph{Cartan connection}) is a geometric structure infinitesimally modelled on a homogeneous space. The curvature of a Cartan geometry vanishes if and only if the Cartan geometry is an \((X,G)\)-structure. 
A \emph{holomorphic Cartan geometry} modelled on \((X,G)\) is a holomorphic \(H\)-bundle \(B \to M\) with a holomorphic connection \(\omega\) on \(B \times^H G\) so that the tangent spaces of \(B\) are transverse to the horizontal spaces of \(\omega\).
A Cartan geometry or locally homogeneous geometric structure is \emph{reductive (parabolic)} if its model is reductive (rational).

If a compact \Kaehler manifold has trivial canonical bundle and a holomorphic parabolic geometry then the manifold has a finite unbranched holomorphic covering by a complex torus and the geometry pulls back to be translation invariant \cite{McKay:2011} p. 3 theorem 1 and p. 9 corollary 2.

\begin{example}
Among complex homogeneous surfaces \((X,G)\) \cite{McKay:2014}, the rational homogeneous varieties are \(A1=\pr{\Proj{2},\PSL{3,\C{}}}\) and
\[
C7=\pr{\Proj{1} \times \Proj{1}, \PSL{2,\C{}} \times \PSL{2,\C{}}}.
\]
Therefore any holomorphic locally homogeneous structure on a complex torus modelled on either of these surfaces is translation invariant.
\end{example}

\begin{theorem}\label{theorem:reductive}
If a compact \Kaehler manifold has a holomorphic reductive Cartan geometry, then the manifold has a finite unbranched holomorphic covering by a complex torus and the geometry pulls back to be translation invariant.
\end{theorem}
\begin{proof}
Holomorphically split \(\LieG=W \oplus \LieH\) for some \(H\)-module \(W\); this \(H\)-module is effective \cite{McKay2013} p. 9 lemma 6.1.
At each point of \(B\), the Cartan connection splits into a 1-form valued in \(W\) and a connection 1-form, say \(\omega=\sigma + \gamma\).
At each point of the total space  \(B\) of the Cartan geometry, the 1-form \(\sigma\) is semibasic, so defines a 1-form \(\bar\sigma\) on the corresponding point of the base manifold, a coframe.
Because \(H\) acts effectively on \(W\), the map \(\bar\sigma\) identifies the total space of the Cartan geometry with a subbundle of the frame bundle of the base manifold \cite{McKay2013} corollary 6.2. 
Hence the Cartan geometry is precisely an \(H\)-reduction of the frame bundle with a holomorphic connection.
The tangent bundle admits a holomorphic connection, so has trivial characteristic classes \cite{Inoue/Kobayashi/Ochiai:1980}.
Therefore the manifold admits a finite holomorphic covering by a complex torus \cite{Inoue/Kobayashi/Ochiai:1980};  without loss of generality assume that the manifold is a complex torus  $T$.
The trivialization of the tangent bundle pulls back to the \(H\)-bundle to be a multiple \(g \sigma\) for some \(g \in \GL{W}\), transforming under \(H\)-action, so defining a holomorphic map \(T \to \GL{W}\!/H\).
But \(\GL{W}\!/H\) is an affine algebraic variety, so admits no nonconstant holomorphic maps from complex tori, so \(T \to \GL{W}\!/H\) is constant, hence without loss of generality is the identity, i.e. \(g\) is valued in \(H\), so there is a holomorphic global section of the bundle on which \(g=1\), trivializing the bundle.
The connection is therefore translation invariant, and so the Cartan connection is translation invariant.
\end{proof}

\begin{example}
Among complex homogeneous surfaces \((X,G)\) \cite{McKay:2014}, the reductive homogeneous surfaces are 
\(A2\), 
\(A3\),
\(C2\),
\(C2'\),
\(C3\),
\(C9\),
\(C9'\),
\(D1\),
\(D1_1\),
\(D1_2\),
\(D1_3\),
\(D1_4\) and
\(D3\).
Therefore any holomorphic locally homogeneous structure on a complex torus modelled on any one of these surfaces is translation invariant.
\end{example}

\begin{proposition}
Suppose that \((X,G)\) is a product of a reductive homogeneous space with  a  rational homogeneous variety.
Then every holomorphic Cartan geometry modelled on \((X,G)\) on any complex torus is translation invariant.
\end{proposition}
\begin{proof}
Write \((X,G)=\pr{X_1 \times X_2, G_1 \times G_2}\) as a product of a reductive homogeneous space \(\pr{X_1,G_1}\) and a rational homogeneous variety \(\pr{X_2,G_2}\).
The splitting of \(X\) into a product splits the tangent bundle of the torus into a product and the canonical bundle into a tensor product.
Since the canonical bundle of the complex torus is trivial, the determinant line bundles of the two factors in our splitting are dual.
The reductive geometry gives a holomorphic connection on the first factor of the splitting of the tangent bundle, so that the determinant line bundle of that factor is trivial.
Therefore the determinant line bundle of the second factor is trivial.
Taking a holomorphic section reduces the structure group of the parabolic part of the geometry to a reductive group \cite{McKay:2011} p. 3, and so the geometry is now reductive so the result follows from theorem~\vref{theorem:reductive}.
\end{proof}

\begin{example}\label{example:products}
Among complex homogeneous surfaces \((X,G)\) \cite{McKay:2014}, those which are a product of a reductive homogeneous curve and a rational homogeneous curve are \(C5\), \(C5'\) and \(C6\).
Therefore any holomorphic locally homogeneous structure on a complex torus modelled on any one of these surfaces is translation invariant.
\end{example}

\section{Lifting}\label{section:lifting}

If a complex homogeneous space \((X,G)\) has underlying manifold \(X\) not simply connected, take the universal covering space \(\tilde{X} \to X\), lift the vector fields that generate the Lie algebra \(\LieG\) of \(G\) to \(\tilde{X}\) and generate an action of a covering group, call it \(\tilde{G}\), acting on \(\tilde{X}\).
Caution: this process neither preserves nor reflects algebraicity.
The developing map \(\delta\) of any \((X,G)\)-structure lifts  to a local biholomorphism \(\tilde{\delta}\) to \(\tilde{X}\).
If all of the deck transformations of \(\fundamentalGroup{X}\) arise as elements of \(\tilde{G}\), then the \((X,G)\)-structure is induced by a unique \((\tilde{X},\tilde{G})\)-structure.
An \((X,G)\)-structure on a complex torus is then translation invariant just when the associated \(\pr{\tilde{X},\tilde{G}}\)-structure is translation invariant.

\begin{example}\label{example:lifts}
From the classification of complex homogeneous surfaces \cite{McKay:2014}, \((X,G)\) has universal covering space \((\tilde{X},\tilde{G})\) with deck transformations carried out by elements of \(\tilde{G}\) for 
any \((X,G)\) among
\(B\beta{1}A1\),
\(B\beta{1}D\),
\(B\beta{1}E\),
\(B\beta{2}'\),
\(B\gamma{2}'\),
\(B\delta{2}'\),
\(C2'\),
\(C5'\), 
\(D1_1\), 
\(D1_2\), 
\(D1_3\), 
\(D1_4\) and 
\(D1_5\).
Therefore the proof of translation invariance of holomorphic \((X,G)\)-structures on complex 2-tori, for these \((X,G)\), reduces to the proof of translation invariance of holomorphic \((\tilde{X},\tilde{G})\)-structures on complex tori, for their universal covering spaces.
\end{example}

A slight modification of this procedure, using proposition~\vref{proposition:enlarge}:
\begin{lemma}
Suppose that we have a morphism \(\pr{\tilde{X},\tilde{G}} \to \pr{X',G'}\) of complex homogeneous spaces from the universal covering space \(\pr{\tilde{X},\tilde{G}} \to \pr{X,G}\) of a complex homogeneous space \((X,G)\).
Suppose that \(\tilde{X} \to X'\) is a local biholomorphism and that \(\tilde{G} \to G'\) has closed image \(\bar{G} \subset G'\).
Suppose that there is no positive dimensional compact complex torus in \(G'/\bar{G}\) acted on transitively by a subgroup of \(G'\).
Suppose that all of the deck transformations of \(\fundamentalGroup{X}\) arise as elements of \(G'\).
Then the developing map of any \((X,G)\)-structure lifts uniquely to the developing map of a unique \(\pr{X',G'}\)-structure.
An \((X,G)\)-structure on a complex torus is then translation invariant just when the associated \(\pr{X',G'}\)-structure is translation invariant.
\end{lemma}

\begin{example}\label{example:lift.and.extend}
Take any complex homogeneous surface \((X,G)\) with universal covering space \(\pr{\tilde{X},\tilde{G}}=B\beta 1\) in the notation of \cite{McKay:2014} (see example~\vref{example:B.beta.1} for the definition).
The inclusion \(B\beta 1 \to B\beta 2\) puts the deck transformations of every quotient \(\pr{\tilde{X},\tilde{G}} \to (X,G)\) into the transformations of a group \(G'\) containing \(\tilde{G}\).
Therefore for all complex homogeneous surfaces \((X,G)\) covered by \(B\beta 1\), every holomorphic \((X,G)\)-structure on any complex torus is translation invariant.
\end{example}

\section{Complex tori of complex dimension 0, 1 or 2}

\begin{theorem}
Every holomorphic locally homogeneous geometric structure on a complex torus of complex dimension 0, 1 or 2 is translation invariant.
\end{theorem}
\begin{proof}
Corollary~\vref{corollary:positive.dim.symmetries} covers any complex torus of dimension 1.
The tricks in examples~\ref{example:B.beta.1} to \ref{example:lift.and.extend} prove translation invariance of all \((X,G)\)-structures for all of the complex homogeneous surfaces \((X,G)\), from the classification \cite{McKay:2014}.
\end{proof}

\section{Conclusion}

It seems that our methods are unable to prove the translation invariance of holomorphic solvable \((X,G)\)-structures on complex tori. We conjecture that all  holomorphic  locally  homogeneous  geometric structures on complex tori are translation invariant.


Moreover, suppose that a complex compact manifold $M$ homeomorphic to a torus  admits a holomorphic \((X,G)\)-structure.
Then we conjecture that  $M$  is biholomorphic to  the quotient \(V/\Lambda\) of a complex vector space $V$  by a lattice $\Lambda$. In other words, nonstandard complex structures on real tori do not admit any holomorphic \((X,G)\)-structure and, more generally, do not admit   any holomorphic rigid geometric structure.

\bibliographystyle{amsplain}
\bibliography{toriGeometry}

\end{document}